\documentclass[12pt,twoside]{article}
\usepackage{amssymb}
\usepackage{amsmath}
\usepackage{amsfonts}
\usepackage{tikz}
\usepackage{centernot}
\usetikzlibrary{arrows}

\usepackage{amsthm}
\usepackage{calc}
\usepackage{tikz-cd}
\usepackage{float}
\usepackage{newlfont}
\usepackage{enumitem}
\usetikzlibrary{calc}

\newcommand{\RN}[1]{%
  \textup{\uppercase\expandafter{\romannumeral#1}}%
}
\date{}

\begin{document}

\vspace*{1.5cm}

\centerline{}

\centerline {\Large{\bf A note on pairs of rings with same prime ideals}} 
\centerline{}

\centerline{\bf {Rahul Kumar\footnote{The author was supported by a grant from UGC India, Sr.
No. 2061440976.} \& Atul
Gaur\footnote{The author was supported by the MATRICS grant from DST-SERB, No. MTR/2018/000707.} }}

\centerline{Department of Mathematics}

\centerline{University of Delhi, Delhi, India.}

\centerline{E-Mail: rahulkmr977@gmail.com; gaursatul@gmail.com}

\centerline{}

\newtheorem{Theorem}{\quad Theorem}[section]

\newtheorem{Corollary}[Theorem]{\quad Corollary}

\newtheorem{Lemma}[Theorem]{\quad Lemma}

\newtheorem{Proposition}[Theorem]{\quad Proposition}
\theoremstyle{definition}

\newtheorem{Definition}[Theorem]{\quad Definition}

\newtheorem{Example}[Theorem]{\quad Example}

\newtheorem{Remark}[Theorem]{\quad Remark}

\begin{abstract}
We study the ring extensions $R \subseteq T$ having the same set of prime ideals provided $Nil(R)$ is a divided prime ideal. Some conditions are given under which no such $T$ exist properly containing $R$. Using idealization theory, the examples are also discussed to strengthen the results.     
\end{abstract}


\section{Introduction}
Throughout this paper, all rings are commutative with nonzero identity and if $R$ is a subring of $T$, then the identity element of $R$ and $T$ is same. By an overring of $R$, we mean a subring of the total quotient ring of $R$ containing $R$. By a local ring, we mean a ring with unique maximal ideal. The symbol $\subseteq$ is used for inclusion, while $\subset$ is used for proper inclusion. We use $T(R)$ to denote the total quotient ring of $R$, $R'$  to denote the integral closure of a ring $R$ in $T(R)$, $Nil(R)$ to denote the set of nilpotent elements of $R$, and $Z(R)$ to denote the set of zero-divisors of $R$. If $I$ and $J$ are $R$-submodules of $T(R)$, then $(I : J) = \{x\in T(R) : xJ\subseteq I\}$. Badawi defined the divided prime ideal of $R$ in $\textup{\cite{bd}}$, as a prime ideal which is comparable to every ideal of $R$ and a ring $R$ is said to be divided if each prime ideal of $R$ is divided, see $\textup{\cite{bd}}$. In $\textup{\cite{badawi3}}$, Anderson and Badawi introduced the notion $\mathcal{H}$ to be the set of all rings $R$ such that $Nil(R)$ is a divided prime ideal and named these rings as $\phi$-rings. They used $\mathcal{H}_0$ to denote the subset of $\mathcal{H}$ such that $Nil(R) = Z(R)$. For further study on these rings, see $\textup{\cite{badawi3}}$, $\textup{\cite{badawi4}}$, $\textup{\cite{badawi}}$, $\textup{\cite{badawi8}}$, $\textup{\cite{badawi6}}$, $\textup{\cite{badawi9}}$, $\textup{\cite{badawi10}}$, $\textup{\cite{badawi7}}$, $\textup{\cite{badawi1}}$, $\textup{\cite{badawi11}}$, $\textup{\cite{badawi12}}$. 

In this paper, we will focus on the ring extensions in class $\mathcal{H}$ having the same set of prime ideals which was extensively studied by Anderson and Dobbs for integral domains, see $\textup{\cite{rahul}}$. Naturally, one may think to study these ring extensions for some larger class. Note that if $R$ is an integral domain, then $R$ is in $\mathcal{H}$. This motivates us to study the ring extensions in class $\mathcal{H}$ having the same set of prime ideals. Note that if $R\in \mathcal{H}$, then there is a ring homomorphism $\phi$ from $T(R)$ to $R_{Nil(R)}$ given by $\phi(r/s) = r/s$, for all $r\in R$ and $s\in R\setminus Z(R)$. Moreover, the restriction of $\phi$ to $R$ is also a ring homomorphism given by $\phi(r) = r/1$, for all $r\in R$, see $\textup{\cite{badawi}}$. Note that if $R\in \mathcal{H}_0$, then $\phi(R) = R$.

In this paper, we discuss examples using this idealization theory. In $\textup{\cite{nagata}}$, Nagata defined a new class of rings, namely idealization of a module. If $R$ is a ring and $M$ is an $R$-module, then the idealization $R(+)M$ is the ring defined as follows: Its additive structure is that of the abelian group $R \oplus M$, and its multiplication is defined by $\left(r_1,m_1\right) \left(r_2,m_2\right) := \left(r_1 r_2,r_1 m_2 + r_2 m_1\right)$, for all $ r_1, r_2\in R$ and $m_1, m_2 \in M$. It will be convenient to view $R$ as a subring  of $R(+)M$ via the canonical injective ring homomorphism that sends $r$ to $\left( r,0\right)$.

\section{Results} 

The ring extensions $R\subseteq T$ for which Spec($R$) = Spec($T$) are studied by Anderson and Dobbs in $\textup{\cite{rahul}}$. They discussed various type of domains $R$ for which there exist domains $T$ satisfying Spec($R$) = Spec($T$), and also given conditions on $R$ for which no such $T$ exists. In Lemma $3.1$, they proved that if $R\subseteq T$ is a ring extension with a common nonzero ideal $I$ which contains a non-zerodivisor of $R$, then $T$ is contained in the total quotient ring of $R$, and indeed $T\subseteq (I: I)$. The result does not seems to be correct. The proof goes wrong if for any nonzero $t \in T$ and for any non-zerodivisor $y \in I$ (of $R$), $ty = 0$. In fact, there is a class of ring extensions for which $ty = 0$ and hence thereby proving that $\textup{\cite[Lemma~3.1]{rahul}}$ is not correct. For example, consider $R = \mathbb{Z}$ and $T = \mathbb{Z}(+) \mathbb{Z}/2\mathbb{Z}$. Then $I = 2\mathbb{Z}$ is a common nonzero ideal with every nonzero element of $I$ is a non-zerodivisor of $R$. But $T\not\subseteq T(R)$. Note that this is true for all domains $R$ with nonzero maximal ideal $M$. In this case, take $T = R(+)R/M$. Then $M$ is a common nonzero ideal but $T\not\subseteq T(R)$. Note that if $ty \neq 0$, then the proof of $\textup{\cite[Lemma~3.1]{rahul}}$ will work. Thus, if we take $y$ to be a non-zerodivisor of $T$, then $\textup{\cite[Lemma~3.1]{rahul}}$ is correct with the same proof. We now state the modified result. 

\begin{Lemma}\label{lem}
Let $R\subseteq T$ be a ring extension with a common nonzero ideal $I$. If $I$ contains a non-zerodivisor of $T$, then $T$ is contained in the total quotient ring of $R$, and indeed $T\subseteq (I: I)$. 
\end{Lemma}

Also, in $\textup{\cite[Remark~3.4(a)]{rahul}}$, Anderson and Dobbs concluded the following: 

Let $R\subset T$ be a ring extension such that Spec($R$) = Spec($T$) and $M$ be the unique maximal ideal of both $R$ and $T$. If $M$ contains a non-zerodivisor of $R$, then $T\subseteq (M: M)$. 

They mentioned that the same proof of $\textup{\cite[Lemma~3.1]{rahul}}$ will work to prove the above result. As $\textup{\cite[Lemma~3.1]{rahul}}$ has been modified, some argument is needed to establish the above result. Note that $M$ contains a non-zerodivisor of $T$. Otherwise $M = Z(T)$ which gives $M = Z(R)$ as Spec($R$) = Spec($T$). This contradicts that $M$ contains a non-zerodivisor of $R$. Thus, the result follows by Lemma $\ref{lem}$. 

Let $R\subset S$ be an extension of integral domains. Then $R$ is a said to be a pseudo-valuation subring (PV, for short) in $S$, see $\textup{\cite{ayache1}}$, if for each $x\in S\setminus R$ and non-unit $a$ of $R$, we have $x^{-1}a\in R$. Note that PV in a domain is a generalization of PVD. We now continue to investigate the properties of ring extensions having the same set of prime ideals initiated in $\textup{\cite{rahul}}$. Here is our first result.

\begin{Theorem}\label{r1}
Let $R\subseteq T\subset U$ be an extension of integral domains such that $Spec(R) = Spec(T)$. Then $R$ is a PV in $U$ if and only if $T$ is a PV in $U$.
\end{Theorem}

\begin{proof}
Assume that $R$ is a PV in $U$. Take $x\in U\setminus T$ and a non-unit $y$ in $T$. Then $y$ is also a non-unit in $R$ as Spec($R$) = Spec($T$). It follows that $x^{-1}y\in R$. Thus, $T$ is a PV in $U$. Conversely, suppose that $T$ is a PV in $U$. Take $x\in U\setminus R$ and $y$ a non-unit in $R$. Then $y$ is also a non-unit in $T$ as Spec($R$) = Spec($T$). If $x\notin T$, then $x^{-1}y\in T$ which is not a unit in $T$ as otherwise $x\in T$. It follows that $x^{-1}y\in R$ and we are done. On the other hand, if $x\in T$, then $x$ is a unit in $T$ as Spec($R$) = Spec($T$). It follows that $x^{-1}y$ is a non-unit in $T$ and so $x^{-1}y\in R$. Thus, $R$ is a PV in $U$.
\end{proof}

Note that the above proposition generalizes $\textup{\cite[Proposition~2.2]{rahul}}$ to PV in a domain. In Proposition $\ref{r3}$, we generalize the same to $\phi$-PVR as $\phi$-PVR generalizes PVD. Recall from $\textup{\cite{badawi}}$ that a ring $R \in \mathcal{H}$ is said to be a $\phi$-pseudo-valuation ring ($\phi$-PVR, for short), if each prime ideal $P$ of $R$ is $\phi$-strongly prime, that is, if $xy\in \phi(P)$ for $x, y\in R_{Nil(R)}$, then either $x\in \phi(P)$ or $y\in \phi(P)$. First, we need the following lemma.    

\begin{Lemma}\label{r2}
Let $R$ be a ring in $\mathcal{H}$. If $T$ is a ring containing $R$ such that $Spec(R) = Spec(T)$, then $T\in \mathcal{H}$.
\end{Lemma}

\begin{proof}
Clearly, $Nil(T) = Nil(R)$ is a prime ideal of $T$. It remains to show that $Nil(R)$ is divided in $T$. Let $I$ be a proper ideal of $T$. Then $I$ is an ideal of $R$ as Spec($R$) = Spec($T$). Thus, $Nil(R)$ is divided in $T$.
\end{proof}

\begin{Proposition}\label{r3}
Let $R\subseteq T$ be a ring extension such that $R\in \mathcal{H}$ and $Spec(R) = Spec(T)$. Then $R$ is a $\phi$-PVR if and only if $T$ is a $\phi$-PVR.
\end{Proposition}

\begin{proof}
Note that $Nil(R)$ is a prime ideal of both $R$ and $T$. It follows that Spec($R/Nil(R)$) = Spec($T/Nil(R)$). Now, if $R$ is a $\phi$-PVR, then $R/Nil(R)$ is a PVD, by $\textup{\cite[Proposition~2.9]{badawi9}}$. It follows that $T/Nil(R)$ is a PVD, by $\textup{\cite[Proposition~2.2]{rahul}}$. Also, $T \in \mathcal{H}$, by Lemma $\ref{r2}$. Thus, by another appeal to $\textup{\cite[Proposition~2.9]{badawi9}}$, $T$ is a $\phi$-PVR. The converse is similar. 
\end{proof} 

A domain $R$ is called a going-down domain if $R\subseteq T$ satisfies the going-down property for each overring $T$ of $R$, see $\textup{\cite{dobbs}}$; and a ring $R$ is called a going-down ring if $R/P$ is a going-down domain for all prime ideals $P$ of $R$, see $\textup{\cite{dobbs1}}$. The next theorem generalizes $\textup{\cite[Corollary~A.4]{rahul}}$ and $\textup{\cite[Proposition~A.1]{rahul}}$ to rings in $\mathcal{H}$. 

\begin{Theorem}\label{r20}
Let $R\in \mathcal{H}$. If $T$ is a ring containing $R$ properly such that $Spec(R) = Spec(T)$, then 
\begin{enumerate}
\item[(i)] $R$ is not a going down ring. 
\item[(ii)] $R$ is divided if and only if $T$ is divided.
\end{enumerate}
\end{Theorem}

\begin{proof}
To prove $(i)$, note that $Nil(R)$ is a prime ideal of both $R$ and $T$. It follows that Spec($R/Nil(R)$) = Spec($T/Nil(R)$). Now, if $R$ is a going down ring, then $R/Nil(R)$ is a going down domain, which is a contradiction by $\textup{\cite[Corollary~A.4]{rahul}}$. 

For $(ii)$, let $R$ be divided and $P$ be any prime ideal of $T$. Then $P$ is a divided prime ideal of $R$. Let $I$ be a proper ideal of $T$. Then $I$ is a proper ideal of $R$ and so is comparable with $P$. Thus, $T$ is divided. Conversely, assume that $T$ is divided. Then $T/Nil(R)$ is divided and so $R/Nil(R)$ is divided, by $\textup{\cite[Proposition~A.1]{rahul}}$. Since $Nil(R)$ is divided, we conclude that $R$ is divided. 
\end{proof}  

A ring $R$ in $\mathcal{H}$ is said to be a $\phi$-chained ring, if for each $x\in R_{Nil(R)}\setminus \phi(R)$, we have $x^{-1}\in \phi(R)$, see $\textup{\cite{badawi6}}$. As a direct consequence of $\textup{\cite[Proposition~3.3]{badawi6}}$, $\textup{\cite[Corollary~3.4]{badawi6}}$, and $\textup{\cite[Proposition~3.8]{rahul}}$, we observe the generalization of $\textup{\cite[Proposition~2.5]{rahul}}$ to rings in $\mathcal{H}$.

\begin{Proposition}\label{r4}
Let $R\in \mathcal{H}$ be a local ring with maximal ideal $M$. If $M$ contains a non-zerodivisor, then the following are equivalent:
\begin{enumerate}
\item[(1)] $R$ is a $\phi$-PVR;
\item[(2)] $(M : M)$ is a $\phi$-chained overring of $R$ with maximal ideal $M$;
\item[(3)] $R$ has a $\phi$-chained overring $V$ such that $Spec(R) = Spec(V)$.
\end{enumerate}
\end{Proposition}

The next corollary is a companion to Theorem $\ref{r20}$ and Proposition $\ref{r4}$ which can be seen as a generalization of $\textup{\cite[Proposition~2.2]{hed}}$ to rings in $\mathcal{H}$.  

\begin{Corollary}\label{r22}
Let $R\in \mathcal{H}$ be a local ring such that the maximal ideal contains a non-zerodivisor of $R$. Then $R$ is a $\phi$-chained ring if and only if $R$ is both a going down ring and a $\phi$-PVR.
\end{Corollary}

\begin{proof}
Let $R$ be a $\phi$-chained ring. Then it is a $\phi$-PVR trivially. Also, $R/Nil(R)$ is a valuation domain, by $\textup{\cite[Theorem~2.7]{badawi3}}$. It follows that $R/Nil(R)$ is a going down domain. Consequently, by $\textup{\cite[Proposition~2.1(a)]{dobbs1}}$, $R$ is going down. 

Conversely,  assume that $R$ is both a going down ring and a $\phi$-PVR. Then by Proposition $\ref{r4}$, it follows that $R$ has a $\phi$-chained overring $V$ such that Spec($R$) = Spec($V$). Thus, $R = V$, by part $(i)$ of Theorem $\ref{r20}$. 
\end{proof}  

The next proposition was proved by Anderson and Dobbs for domains, see $\textup{\cite[Proposition~3.5(b)]{rahul}}$. Note that the same is true for rings also as the proof of $\textup{\cite[Proposition~3.5(b)]{rahul}}$ do not requires the ring to be a domain. Thus, the proof of next proposition follows mutatis mutandis from the proof of $\textup{\cite[Proposition~3.5(b)]{rahul}}$. 

\begin{Proposition}\label{r5}
Let $R\subseteq T$ be rings such that $Spec(R) = Spec(T)$. Then $R_{P} = T_{T\setminus P}$ for each non-maximal prime ideal $P$ of $R$.
\end{Proposition}

The next proposition generalizes $\textup{\cite[Corollary~3.20(2)]{rahul}}$ to rings in $\mathcal{H}$. 

\begin{Proposition}\label{r6}
Let $R\subseteq T$ be a ring extension such that $T\in \mathcal{H}$ and $R$ is local with maximal ideal $M$. Assume that $M$ is a finitely generated ideal of $R$. Then $Spec(R) = Spec(T)$ if and only if $M\in Spec(T)$. 
\end{Proposition}

\begin{proof}
The ``only if'' assertion is trivial. For the converse part, assume that $M\in \mbox{Spec}(T)$. Then $Nil(T) = Nil(R)$. Also, $R/Nil(R)$ is local with finitely generated maximal ideal $M/Nil(R)$ and $T/Nil(R)$ is a domain. It follows that Spec($R/Nil(R)$) = Spec($T/Nil(R)$), by $\textup{\cite[Corollary~3.20]{rahul}}$. Thus, the result holds. 
\end{proof}

\begin{Remark}\label{r7}
Let $R$ be a Pr\"ufer domain which is not a field. Then there is no ring $T$ which contains $R$ properly such that Spec($R$) = Spec($T$), see $\textup{\cite[Remark~3.17(b)]{rahul}}$. Note that this can be extended to $\phi$-Pr\"ufer rings also (A ring $R\in \mathcal{H}$ is said to be a $\phi$-Pr\"ufer ring, see $\textup{\cite{badawi3}}$, if $\phi(R)$ is a Pr\"ufer ring). More precisely, if $R$ is a $\phi$-Pr\"ufer ring such that $Nil(R)$ is not a maximal ideal of $R$ and $T$ is a ring containing $R$ such that Spec($R$) = Spec($T$), then $R = T$. To see this, first note that $Nil(R) = Nil(T)$. It follows that both $R/Nil(R)$ and $T/Nil(R)$ have the same set of prime ideals. As by $\textup{\cite[Theorem~2.6]{badawi3}}$, $R/Nil(R)$ is a Pr\"ufer domain,  we have $R/Nil(R)  = T/Nil(R)$ and so $R = T$.
\end{Remark}

In the next proposition, we investigate the overrings of $R \in \mathcal{H}_0$ having the same set of prime ideals. Moreover, it can be seen as a generalization of $\textup{\cite[Corollary~3.21]{rahul}}$. 

\begin{Proposition}\label{r8}
Let $R\in \mathcal{H}_0$ be a local ring with maximal ideal $M$ such that $R'$ is a Pr\"ufer ring. Let $T$ be an overring of $R$. Then $Spec(R) = Spec(T)$ if and only if $M\in Spec(T)$.
\end{Proposition}

\begin{proof}
It is easy to see that both $R'$ and $T$ are in $\mathcal{H}_0$. Moreover, by $\textup{\cite{badawi}}$, $Nil(T) = Nil(R') = Nil(R)$ and $T(R/Nil(R)) = T(R)/Nil(R)$. It follows that $T/Nil(R)$ is an overring of $R/Nil(R)$. Moreover, by $\textup{\cite[Lemma~2.8]{badawi4}}$, $(R/Nil(R))' = R'/Nil(R)$. Now, by $\textup{\cite[Corollary~2.10]{badawi3}}$, $(R/Nil(R))'$ is a Pr\"ufer domain. Consequently, Spec($R/Nil(R)$) = Spec($T/Nil(R)$) if and only if $M/Nil(R)\in \mbox{Spec}(T/Nil(R))$, by $\textup{\cite[Corollary~3.21]{rahul}}$. Hence, the result follows.
\end{proof} 

\begin{Remark}
In $\textup{\cite{rahul}}$, Anderson and Dobbs raised the following question: given a ring $R$, how does one can obtain all overrings $T$ of $R$ which satisfies Spec($R$) = Spec($T$). They answered the question in $\textup{\cite[Theorem~3.5]{rahul}}$ for domains. However, observe that the proof of $\textup{\cite[Theorem~3.5]{rahul}}$ works for rings also. 
\end{Remark}

We now discuss an example that gives a class of ring extensions having the same set of prime ideals. 

\begin{Example}\label{r11}
Let $B$ be a ring of the form $K + M$, where $K$ is a field and $M$ is a maximal ideal of $B$. Assume that $D$ is a subring of $K$, and
$A = D + M$. Set $U = A(+)T(B)$ and $V = B(+)T(B)$. Then Spec($U$) = Spec($V$) if and only if $V$ is local and $D$ is a field. To see this, note that if Spec($U$) = Spec($V$), then Spec($A$) = Spec($B$), by $\textup{\cite[Theorem~25.1(3)]{huckaba}}$. It follows that $B$ is local and $D$ is a field, by $\textup{\cite[Corollary~3.11]{rahul}}$. Consequently, $V$ is local, by another appeal to $\textup{\cite[Theorem~25.1(3)]{huckaba}}$. Conversely, assume that $V$ is local and $D$ is a field. Then $B$ is local, by $\textup{\cite[Theorem~25.1(3)]{huckaba}}$. It follows that Spec($A$) = Spec($B$), by $\textup{\cite[Corollary~3.11]{rahul}}$. Consequently, by another application of $\textup{\cite[Theorem~25.1(3)]{huckaba}}$, Spec($U$) = Spec($V$).
\end{Example}

The next proposition gives a class of rings $R$ in $\mathcal{H}_0$ which does not admit any ring $T$ properly containing $R$ and having the same set of prime ideals. For domains, it was observed in $\textup{\cite[pg.~372]{rahul}}$.  

\begin{Proposition}\label{r12}
Let $R\in \mathcal{H}_0$ be an integrally closed local ring such that $R\neq T(R)$. If the maximal ideal $M$ of $R$ is finitely generated, then there is no ring $T$ properly containing $R$ such that $Spec(R) = Spec(T)$.
\end{Proposition}

\begin{proof}
If possible, suppose there exists a ring $T$ properly containing $R$ such that Spec($R$) = Spec($T$). Then we have $Nil(T) = Nil(R)$ and Spec($R/Nil(R)$) = Spec($T/Nil(R)$). Moreover, by $\textup{\cite{badawi}}$ and $\textup{\cite[Lemma~2.8]{badawi4}}$, we conclude that $T(R/Nil(R)) = T(R)/Nil(R)$ and $(R/Nil(R))' = R'/Nil(R)$, respectively. By hypothesis, it follows that $R/Nil(R)$ is an integrally closed local domain which is not a field. Moreover, the maximal ideal $M/Nil(R)$ of $R/Nil(R)$ is finitely generated. Consequently, we have a contradiction, by last paragraph of $\textup{\cite[pg.~372]{rahul}}$.
\end{proof}

The next example shows that the condition on $M$ of being finitely generated can not be removed. 

\begin{Example}
Let $F\subset K$ be a purely transcendental field extension. Take a valuation domain $B = K + M$ with unique nonzero maximal ideal $M$. Set $R = A(+)T(B)$ and $T = B(+)T(B)$ where $A = F + M$. Then it is easy to see that $Z(R) = Nil(R) = Nil(T) = \{0\}(+)T(B)$. Let $(0, x)\in Nil(R)$ and $(y, z)\in R\setminus Nil(R)$. Then $y\neq 0$ and $(0, x) = (y, z)(0, x/y)$. It follows that $R\in \mathcal{H}_0$. Since $A$ is an integrally closed domain, $R$ is an integrally closed ring, by $\textup{\cite[Theorem~25.6]{huckaba}}$. Also, by $\textup{\cite[Corollary~25.5(3)]{huckaba}}$, $R\neq T(R)$. Since $M$ is not finitely generated, $M(+)T(B)$ is not finitely generated. Note that $M(+)T(B)$ is the unique maximal ideal of $R$, by $\textup{\cite[Theorem~25.1(3)]{huckaba}}$. However, by Example $\ref{r11}$, Spec($R$) = Spec($T$).
\end{Example} 

Let $C(R)$ denotes the complete integral closure of $R$ in $T(R)$. In the next theorem, we show that the ring extensions in $\mathcal{H}_0$ having the same set of prime ideals will have the same complete integral closure. This generalizes $\textup{\cite[Proposition~3.15]{rahul}}$ for rings in $\mathcal{H}_0$. 

\begin{Theorem}\label{r13}
Let $R \in \mathcal{H}_0$ be such that $R\neq T(R)$. If $T$ is a ring containing $R$ such that $Spec(R) = Spec(T)$, then $C(R) = C(T)$. 
\end{Theorem}

\begin{proof}
It is easy to conclude that $Nil(R) = Nil(T)$, $Z(T) = Z(R)$ and Spec($R/Nil(R)$) = Spec($T/Nil(R)$). Also, $T(R/Nil(R)) = T(R)/Nil(R)$, by $\textup{\cite{badawi}}$, and so $R/Nil(R)$ is not a field. By $\textup{\cite[Proposition~3.15]{rahul}}$, it follows that $C(R/Nil(R)) = C(T/Nil(R))$. Now, by Lemma $\ref{r2}$, $T\in \mathcal{H}_0$. Consequently, we have $C(R)/Nil(R) = C(T)/Nil(R)$, by $\textup{\cite[Lemma~2.8]{badawi4}}$. Thus, $C(R) = C(T)$.
\end{proof}

The following corollary is an immediate consequence of Theorem $\ref{r13}$ when $R = C(R)$.  

\begin{Corollary}\label{r14}
Let $R\in \mathcal{H}_0$ be a completely integrally closed ring such that $R\neq T(R)$. Then there is no ring $T$ properly containing $R$ such that $Spec(R) = Spec(T)$. 
\end{Corollary}

In $\textup{\cite[Lemma~2.4]{rahul}}$, Anderson and Dobbs proved the next result for local domains which are not fields. Here, we generalize the same for rings. In fact, the proof follows mutatis mutandis from the proof of $\textup{\cite[Lemma~2.4]{rahul}}$. Note that an ideal $I$ of a ring $R$ is said to be divisorial if $I = (R : (R: I))$. 

\begin{Theorem}\label{r17}
Let $R$ be a local ring with maximal ideal $M$. If $M$ contains a non-zerodivisor of $R$, then the following hold:
\begin{enumerate}
\item[(i)] $(R : M) = (M : M)$ if and only if $M$ is not a principal ideal of $R$. Moreover, if $M = Rr$ for some $r\in R$, then $(R : M) = Rr^{-1}$ and $(M : M) = R$.
\item[(ii)] If $M$ is not a divisorial ideal of $R$, then $(R : M) = (M : M)$.
\end{enumerate}
\end{Theorem}

The next proposition generalizes $\textup{\cite[Proposition~3.23]{rahul}}$. 

\begin{Proposition}\label{r18}
Let $R$ be a local ring with maximal ideal $M$ such that $M$ contains a non-zerodivisor of $R$. Then $(M : M)\neq R$ if and only if $M$ is a non-principal divisorial ideal of $R$.
\end{Proposition}

\begin{proof}
The ``only if'' assertion follows mutatis mutandis from the proof of $\textup{\cite[Proposition~3.23]{rahul}}$ and ``if'' assertion follows from Theorem $\ref{r17}$.  
\end{proof}

The proof of the next corollary follows directly from Proposition $\ref{r18}$ and Lemma $\ref{lem}$. 

\begin{Corollary}\label{r19}
Let $R$ be a local ring such that the maximal ideal of $R$ contains a non-zerodivisor of $R$. Assume that the maximal ideal of $R$ is either principal or a non-divisorial ideal of $R$. Then there is no ring $T$ properly containing $R$ such that $Spec(R) = Spec(T)$.
\end{Corollary}

It is easy to see that $\textup{\cite[Corollary~3.30]{rahul}}$ is not true for fields as if $R$ is a field, then $R'$ is not of rank one. However, if we exempt $R$ to be a field, then $R'$ is a discrete (rank 1) valuation domain. In Proposition $\ref{r23}$, we generalize the same to rings in $\mathcal{H}_0$. First, we recall some definitions. An ideal of a ring $R$ is said to be a nonnil ideal if $I\not \subseteq Nil(R)$. A ring $R\in \mathcal{H}$ is called a nonnil-Noetherian ring if every nonnil ideal of $R$ is finitely generated, see $\textup{\cite{badawi10}}$. Recall from $\textup{\cite{badawi4}}$ that a ring $R\in \mathcal{H}$ is said to be a discrete $\phi$-chained ring if $R$ is a $\phi$-chained ring with at most one nonnil prime ideal and every nonnil ideal of $R$ is a principal ideal. 

\begin{Proposition}\label{r23}
Let $R\in \mathcal{H}_0$. If $R$ is a nonnil-Noetherian $\phi$-PVR, then dim($R)\leq 1$ and $R'$ is a discrete $\phi$-chained ring. 
\end{Proposition}

\begin{proof}
Note that $(R/Nil(R))' = R'/Nil(R)$, by $\textup{\cite[Lemma~2.8]{badawi4}}$. Also, by $\textup{\cite[Theorem~2.2]{badawi10}}$ and $\textup{\cite[Proposition~2.9]{badawi9}}$, $R/Nil(R)$ is a Noetherian PVD. It follows that $dim(R) \leq 1$ and $R'/Nil(R)$ is a discrete valuation domain, by $\textup{\cite[Corollary~3.30]{rahul}}$. Thus, the result holds, by $\textup{\cite[Lemma~2.9]{badawi4}}$.
\end{proof}

Recall from $\textup{\cite{m}}$ that a ring $R$ is said to be a weakly finite-conductor ring if $Ra\cap Rb$ is a finitely generated ideal for all $a, b\in R$.  We end this paper with the generalization of $\textup{\cite[Corollary~A.3]{rahul}}$ to rings in $\mathcal{H}_0$. Note that $\textup{\cite[Corollary~A.3]{rahul}}$ fails to hold when $R\subset T$ is a field extension such that $T$ is not a finitely generated $R$-module. Therefore, the exemption of $R$ to be a field is required in the statement of $\textup{\cite[Corollary~A.3]{rahul}}$. 

\begin{Proposition}\label{r24}
Let $R\subset T$ be a ring extension such that $R\in \mathcal{H}_0$ is a weakly finite-conductor ring, $R\neq T(R)$, and $Spec(R) = Spec(T)$. Let $M$ be the maximal ideal of $R$. Then both $M$ and $T$ are finitely generated $R$-modules.
\end{Proposition}

\begin{proof}
Note that $R/Nil(R)$ is not a field as $T(R/Nil(R)) = T(R)/Nil(R)$, by $\textup{\cite{badawi}}$. Also, it is easy to see that $R/Nil(R)$ is a finite-conductor domain with maximal ideal $M/Nil(R)$. Thus, $M/Nil(R)$ and $T/Nil(R)$ are finitely generated $R/Nil(R)$-module, by $\textup{\cite[Corollary~A.3]{rahul}}$. Since $R\neq T(R)$, $M$ is a nonnil ideal of $R$. Now, the proof follows mutatis mutandis from the proof of $\textup{\cite[Lemma~2.4]{badawi3}}$.
\end{proof}

\end{document}